\documentclass[a4paper,12pt]{article}
\usepackage{cmap}
\usepackage[cp1251]{inputenc}
\usepackage[english]{babel}
\usepackage[left=2cm,right=2cm,top=2cm,bottom=2cm]{geometry}
\usepackage{amssymb}
\usepackage{amsmath, amsthm}
\theoremstyle{plain}
\newtheorem{theorem}{Theorem}

\theoremstyle{definition}
\newtheorem{example}{Example}

\newtheorem{question}{Question}

\theoremstyle{remark}

\begin{document}

\begin{center}\Large
%\textbf{On Groups with Conjugate-Permutable Subgroups}
\textbf{On ECP-Groups}

\normalsize

\smallskip
Viachaslau I. Murashka

 \{mvimath@yandex.ru\}

Faculty of Mathematics and Technologies of Programming,

 Francisk Skorina Gomel State University,  Gomel 246019, Belarus\end{center}

\begin{abstract}
  According to T. Foguel  a subgroup $H$ of a group $G$ is called conjugate-permutable if $ HH^x=H^xH$ for every $x\in G$. Mingyao Xu and Qinhai Zhang studied finite groups with every subgroup conjugate-permutable (ECP-groups) and asked three questions about them. We gave the answers on these questions. In particular, every group of exponent 3 is ECP-group, there exist non-regular ECP-$3$-groups and the class of all finite ECP-groups is neither formation nor variety. \end{abstract}


\textbf{AMS}(2010). 20D15  (Primary)  20D35 (Secondary).

\section*{Introduction}

Recall    \cite{Foguel1997} that a subgroup $H$  of a group $G$   is called conjugate-permutable if $HH^x=H^xH$ for every $x\in G$. A group is called ECP-group \cite{Xu2005} (resp. CCP-group \cite{Foguel1999}) if every its (resp. cyclic) subgroup is conjugate-permutable. The class of all ECP-groups is rather interesting. It includes the classes of groups all whose subgroups are 2-subnormal and all whose subgroups are permutable.
In \cite{Xu2005} it was proved that  a finite group $G$ is an ECP-group if and only if $G$ is nilpotent and every Sylow $p$-subgroup of $G$ is an ECP-group. This result reduces the study of ECP-groups to the study of ECP-$p$-groups. In this case
Mingyao Xu and Qinhai Zhang asked the following questions:

\begin{question}[{\cite[Question 3.9]{Xu2005}}]\label{q1}
  Is every finite group of exponent 3 an ECP-group?
\end{question}

\begin{question}[{\cite[Question 3.12]{Xu2005}}]\label{q2}
  For $p=3$, is every finite ECP-$p$-group regular?
\end{question}

\begin{question}[{\cite[Question 3.13]{Xu2005}}]\label{q3}
  Does the class of all finite ECP-groups form a variety or a formation?
\end{question}

The aim of this paper is to give the answers on them.

\section{Main Results}

The positive answer on Question \ref{q1} gives

\begin{theorem}
  Every group of exponent 3 is an ECP-group.
\end{theorem}

\begin{proof}
  Let $H$ be a subgroup of a group  $G$ of exponent 3, $x\in G$ and $a, b\in H$. Let $y, z\in G$. Then $(yz)(yz)(yz)=(yzy)(zyz)=1$. Hence $yzy=z^{-1}y^{-1}z^{-1}$ holds for all $y, z\in G$.  Now
   \begin{align*}ab^x&=ax^{-1}bx=(x^{-1}x)ax^{2}bx=x^{-1}(xax)(xbx)\\
   &=x^{-1}(a^{-1} x^{-1}a^{-1})(b^{-1}x^{-1}b^{-1})=x^{-1}a^{-1} (x^{-1}a^{-1}b^{-1}x^{-1})b^{-1}\\
   &=x^{-1}a^{-1}(ba x ba)b^{-1}=(a^{-1}ba)^x(bab^{-1}).
      \end{align*}
By the same arguments one can show that $ab^{x^2}=(a^{-1}ba)^{x^2}(bab^{-1})$. From $x^3=1$ it follows that $a^xb=(a^{-1}ba)(bab^{-1})^x$.  It means that $HH^x\subseteq H^xH$ and $H^xH\subseteq HH^x$. So $HH^x=H^xH$. It means that $H$ is  a conjugate-permutable subgroup of $G$. Thus $G$ is an ECP-group.
\end{proof}

The negative answer on Question \ref{q2} gives

\begin{theorem}
  The classes of all finite CCP-groups and ECP-groups are not closed under taking direct products. Hence they are not formations and not varieties.
\end{theorem}

\begin{proof}
  Let
$$G=\langle a, b, c, d\mid a^{27}=c^{27}=b^9=d^9=[a, c]=[a, d]=[b,c]=[b, d]=a^ba^{-4}=c^dc^{-4}=1\rangle. $$
Let $H_1=\langle a, b\rangle$ and $H_2=\langle c, d\rangle$.  Then $G=H_1\times H_2$ and $H_1\simeq H_2$. Note that $H_i$ is a product of two cyclic groups. Hence all its subgroups are permutable \cite[Satz 15]{Huppert1953}. So $H_i$ is a ECP-group.
Let \begin{multline*}K=\langle k\rangle=\langle a^3b^2c^3d\rangle=\\\{e, a^{3}b^{2}c^{3}d,a^{15}b^{4}c^{24}d^{2},a^{9}b^{6}c^{9}d^{3},a^{12}b^{8}c^{12}d^{4}, a^{24}b^{10}c^{6}d^{5}, a^{18}b^{12}c^{18}d^{6}, a^{21}b^{14}c^{21}d^{7}, a^{6}b^{16}c^{15}d^{8}\}.
\end{multline*}
 and $x=abcd$. Then
\begin{multline*}
  K^x=\langle k^x\rangle=\langle a^{15}b^2c^9d\rangle=\\
  \{e, a^{15}b^{2}c^{9}d,a^{21}b^{4}c^{18}d^{2},a^{18}b^{6}c^{0}d^{3},a^{6}b^{8}c^{9}d^{4}, a^{12}b^{10}c^{18}d^{5}, a^{9}b^{12}c^{0}d^{6}, a^{24}b^{14}c^{9}d^{7}, a^{3}b^{16}c^{18}d^{8}\}.
\end{multline*}
Note that $y=(a^{15}b^{2}c^{9}d)(a^{3}b^{2}c^{3}d)=a^0b^4c^3d^2\in K^xK$. Note that $c^9\in \mathrm{Z}(G)$. Assume that $y\in KK^x$. Note that if we write $y=k^i(k^x)^j$  in the form $a^{q_1}b^{q_2}c^{q_3}d^{q_4}$, then $q_4\equiv i+j \mod 9$ and $q_3\equiv q_5 \mod 9$ where $q_5$ is a power of $c$ in $k^i$.   It means that $i\equiv 2-j \mod 9$ and $q_5\equiv 3 \mod 9$. Hence $(i, j)\in\{(1,1), (4, 7), (7,4)\}$. But $q_3\neq 3$ for  $(i, j)\in\{(1,1), (4, 7)\}$. Note that $(a^{21}b^{14}c^{21}d^{7})(a^{6}b^{8}c^{9}d^{4})=a^{18}b^4c^3d^2\neq a^0b^4c^3d^2$. Thus $y\not\in KK^x$, the contradiction. Hence $K$ is a not conjugate permutable subgroup of $G$.
\end{proof}

From this result it follows that the first statement of \cite[Lemma 2.11]{Foguel1999} is false.

The negative answer on Question \ref{q3} gives

\begin{theorem}
  There exists a finite non-regular 3-group.
\end{theorem}

\begin{proof}
  Here we describe a brute force test for a group to be an ECP-group in GAP. The function ``ArePermutableSubgroups'' is taken from the package ``permut'' \cite{ABB014}. Note that if a subgroup is conjugate-permutable, then all its conjugates also have this property.

\medskip

IsECPGroup:=function(G)

local S,b,a;

\quad S:=ConjugacyClassesSubgroups(G);

\quad for a in S do

\quad\quad for b in ConjugateSubgroups(G,a[1]) do

\quad\quad\quad if (not ArePermutableSubgroups(b,a[1])) then

\quad\quad\quad\quad return false;

\quad\quad\quad fi;

\quad\quad od;

\quad od;

\quad return true;

end;;

\medskip

If in the previous algorithm we replace the line ``S:=ConjugacyClassesSubgroups(G);'' by the line ``S:=Filtered(ConjugacyClassesSubgroups(G),x-$>$IsCyclic(x[1]));'', then we will obtain the function which tests wether or note a group is a CCP-group.

 Let $G=\langle a, b\rangle$ where
 \begin{align*}
  a=&(1,2,6,5,9,18,15,24,37)(3,20,70,12,41,79,29,62,53)(4,23,57,14,44,17,31,8,36)\\
  &(7,33,66,21,54,27,42,71,48)(10,58,78,26,73,51,47,80,68)(11,61,16,28,75,34,49,40,55)\\
  &(13,63,56,30,22,72,50,43,35)(19,67,25,39,77,46,60,81,65)(32,64,38,52, 76,59,69,45,74),
\\
 b=&(1,3,10,15,29,47,5,12,26)(2,7,19,24,42,60,9,21,39)(4,11,25,31,49,65,14,28,46)\\
 &(6,16,32,37,55,69,18,34,52)(8,20,38,44,62,74,23,41,59)(13,27,45,50,66,76,30,48,64)\\
 &(17,33,51,57,71,78,36,54,68)(22,40,58,63,75,80,43,61,73)(35,53,67,72,79,81,56,70,77).
 \end{align*}
Then $G$ is a ECP-group (it can be checked with the function ``IsECPGroup(G)'' in GAP). Note that the exponent of the derived subgroup of $G$ is 3 (it can be checked with the function ``Exponent(DerivedSubgroup(G))'' in GAP). Hence if $G$ is 3-regular, then $(ab)^3a^{-3}b^{-3}=()$. But
\begin{align*}
(ab)^3a^{-3}b^{-3}=&  (1,5,15)(2,9,24)(3,12,29)(4,14,31)(6,18,37)(7,21,42)(8,23,44)(10,26,47)\\
&(11,28,49)(13,30,50)(16,34,55)(17,36,57)(19,39,60)(20,41,62)(22,43,63)\\
&(25,46,65)(27,48,66)(32,52,69)(33,54,71)(35,56,72)(38,59,74)(40,61,75)\\
&(45,64,76)(51,68,78)(53,70,79)(58,73,80)(67,77,81)
\end{align*}
The Id of this group in the library of SmallGroups \cite{Besche2002} is [81, 10].
\end{proof}

\begin{example}
  In the general classes of CCP-groups and ECP-groups are different. Let $ G=\langle a, b, c, d\rangle$ where
  \begin{align*}
  a=&(2,3)(6,7)(9,12)(13,16)(17,32,20,29)(18,30,19,31)(21,25,24,28)(22,27,23,26),\\
  b=&(1,17)(2,18)(3,19)(4,20)(5,21)(6,22)(7,23)(8,24)\\
  &(9,25)(10,26)(11,27)(12,28)(13,29)(14,30)(15,31)(16,32),\\ c=&(1,2)(3,4)(5,7)(6,8)(9,10)(11,12)(13,15)(14,16)\\
  &(17,18)(19,20)(21,23)(22,24)(25,26)(27,28)(29,31)(30,32),\\ d=&(1,8,4,5)(2,7,3,6)(9,16,12,13)(10,15,11,14)\\
  &(17,28,20,25)(18,27,19,26)(21,32,24,29)(22,31,23,30).
  \end{align*}
It can be checked with functions ``IsCCPGroup(G)'' and ``IsECPGroup(G)'' that $G$ is CCP-group not-ECP-group. The Id of this group in the library of SmallGroups \cite{Besche2002} is [128, 1760].
\end{example}

\begin{theorem}
Let $h$ be an $n$-left-Engel element of  a finite conjugate-permutable subgroup   $H$ of a group $G$. Let $t$ be a number of not necessary different prime divisors in a decomposition of $|H|$. Then $h$ is   $(n+t+1)$-left-Engel element of $G$.
%If $H$ is abelian, then $h$ is   $(n+1)$-left-Engel element of $G$
\end{theorem}

\begin{proof}
  Let prove that $[x,_{t+1} h]\in H$ for every $x\in G$. Note that $[x,_{i} h]^{-1}=[h, [x,_{i-1} h]]=h^{-1}h^{[x,_{i-1} h]}\in HH^{[x,_{i-1} h]}$. Hence $[x,_{i} h]\in HH^{[x,_{i-1} h]}$. Therefore $HH^{[x,_{i} h]}$ is a  subgroup of $ HH^{[x,_{i-1} h]}$ for every $i\geq 1$ (recall that $[x,_0 h]=x)$. Hence  $|H\cap H^{[x,_{i-1} h]}|$ divides $|H\cap H^{[x,_{i} h]}|$ for every $i\geq 1$. 
  
  Assume that $[x,_{i} h]\not\in H$ for all $0\leq i< t$. Then $|H:H\cap H^{[x,_{t-1} h]}|$ is a prime. Hence $H$ is a maximal conjugate-permutable subgroup of $H H^{[x,_{t-1} h]}$. It means that $ H\trianglelefteq H H^{[x,_{t-1} h]}$. Hence $H^{[x,_{t} h]}=H$. Now $[x,_{t+1} h]\in H$. It means that   $[x,_{n+t+1} h]=[[x,_{t+1} h],_n h]=1$.
\end{proof}

{\small\bibliographystyle{siam}
\bibliography{ECP}}

\end{document}